\documentclass{amsart}

\usepackage{graphicx}

\usepackage{amsmath}

\usepackage{amsfonts}

\usepackage{amssymb}

\newtheorem{theorem}{Theorem} [section]

\newtheorem{thm}[theorem]{Theorem}

\newtheorem{cor}[theorem]{Corollary}

\newtheorem{lemma}[theorem]{Lemma}

\newtheorem{prop}[theorem]{Proposition}

\begin{document}


\title{Intrinsic Linking and Knotting in Tournaments}

\author{Thomas Fleming}
\author{Joel Foisy}

\begin{abstract}

A directed graph $G$ is \emph{intrinsically linked} if every embedding of that graph contains a non-split
link $L$, where each component of $L$ is a consistently oriented cycle in $G$. A \emph{tournament} is a
directed graph where each pair of vertices is connected by exactly one directed edge.

We consider intrinsic linking and knotting in tournaments, and study the minimum number of vertices required
for a tournament to have various intrinsic linking or knotting properties. We produce the following bounds:
intrinsically linked ($n=8$), intrinsically knotted ($9 \leq n \leq 12$), intrinsically 3-linked ($10 \leq n
\leq 23$), intrinsically 4-linked ($12 \leq n \leq 66$), intrinsically 5-linked ($15 \leq n \leq 154$),
intrinsically $m$-linked ($3m \leq n \leq 8(2m-3)^2$), intrinsically linked with knotted components ($9 \leq
n \leq 107$), and the disjoint linking property ($12 \leq n \leq 14$).

We also introduce the \emph{consistency gap}, which measures the difference in the order of a graph required
for intrinsic $n$-linking in tournaments versus undirected graphs. We conjecture the consistency gap to be
non-decreasing in $n$, and provide an upper bound at each $n$.
\end{abstract}

\maketitle

\section{Introduction}

A graph is called \emph{intrinsically linked} if every embedding of that graph in $S^3$ contains cycles that
form a non-split link, and \emph{intrinsically knotted} if every embedding of that graph contains a cycle
that is a non-trivial knot. These properties were first studied by Sachs \cite{sachs} and by Conway and
Gordon \cite{cg}.

Researchers have studied variations of these properties, such as requiring every embedding of the graph to
contain cycles that form a non-split $n$-component link \cite{ffnp}, a non-split link where one of more of
the components are non-trivial knots \cite{flapan} \cite{flem}, or even more complex structures \cite{nikk2}
\cite{REU}.

A directed graph is \emph{intrinsically $n$-linked as a directed graph} (or \emph{intrinsically knotted as a
directed graph}) if every embedding of that graph in $S^3$ contains cycles that form a non-split $n$
component link (or non-trivial knot), and the edges that comprise each cycle of the link (or knot) have a
consistent orientation.  Examples of intrinsically 2-linked directed graphs are known \cite{FHR}, and
intrinsically knotted, intrinsically 3-linked and 4-linked directed graphs have also been constructed
\cite{flemfoisy}.

In this work, we focus on a subset of directed graphs known as tournaments.   A \emph{tournament} on $n$
vertices is a directed graph with exactly one directed edge between each pair of vertices.  Equivalently, a
tournament on $n$ vertices is $K_n$ with a choice of orientation for each edge. We then ask, given an
intrinsic linking property, what is the smallest $n$ such that there exists a tournament on $n$ vertices
having that property?  In Section 2, we study intrinsic linking and have the precise answer $n=8$.   That is,
while $K_6$ is intrinsically linked, no tournament on 7 or fewer vertices is intrinsically linked as a
digraph, and there exists a tournament on 8 vertices that is intrinsically linked as a digraph.

For intrinsic knotting we have the bounds $9 \leq n \leq 12$ in Section 3, and for the disjoint linking
property we have $12 \leq n \leq 14$ in Section 9.  In Sections 4, 5 and 6 we find $10 \leq n \leq 23$ for
intrinsic 3-linking, $12 \leq n \leq 66$ for intrinsic 4-linking and $15 \leq n \leq 154$ for intrinsic
5-linking.  We address $m$-linking for $m>5$ in Section \ref{n_link_section} and have $3m \leq n \leq
8(2m-3)^2$.  In Section 8, we construct a tournament that contains a non-split link where at least one of the
components is a non-trivial knot, obtaining the bounds $9 \leq n \leq 107$ for this property.

The construction of the 4-linked tournament in Section 5 is similar in spirit to that of the 4-linked
directed graph in \cite{flemfoisy} but is substantially less complicated and requires far fewer vertices and
edges. Section 6 extends this construction to produce a 5-linked tournament.
Adapting techniques of Flapan, Mellor and Naimi \cite{flapan}, we are able to demonstrate $n$-linked
tournaments for all $n$. In very recent work, Mattman, Naimi and Pagano independently used similar techniques
to construct examples of intrinsically $n$-linked complete symmetric directed graphs \cite{mnp}.

Given a tournament, we may ignore the edge orientations and consider it as an undirected complete graph.
This graph may be intrinsically $n$-linked even if the tournament is not intrinsically $n$-linked as a
directed graph.  Thus the requirement that the components of the non-split link have a consistent orientation
for an intrinsically $n$-linked digraph is restrictive and appears to require a larger and more complex graph
to satisfy.

We introduce the \emph{consistency gap} to measure this difference and denote it as $cg(n)$. We define $cg(n)
= m' - m$ where $K_m$ is the smallest intrinsically $n$-linked complete graph, and $m'$ is the number of
vertices in the smallest tournament that is intrinsically $n$-linked as a directed graph.  In Section 10, we
show that $cg(2)=2$ and provide a bound on $cg(n)$ for all $n$.

\section{Intrinsic Linking in Tournaments}

\begin{prop} No tournament on 6 vertices is intrinsically linked as a directed graph.
\label{K_6_nIL_prop}
\end{prop}

\begin{proof}
An orientation of $K_6$ gives a directed graph on 6 vertices with 15 edges.  By Corollary 3.10 of \cite{FHR},
any directed graph on 6 vertices with 23 or fewer edges is not intrinsically linked as a directed graph.
\end{proof}

\begin{thm}
No tournament on 7 vertices is intrinsically linked as a directed graph. \label{K7_NIL_thm}
\end{thm}

\begin{proof}
Let $T$ be a tournament on 7 vertices.  We will break the proof in to cases based on the maximum in degree of
a vertex in $T$. We will check only the cases of max in degree equal to 6, 5, 4, and 3; as for the other
cases we may consider the vertex of maximum out degree, and apply the same arguments.

\begin{figure}
\includegraphics[scale=0.45]{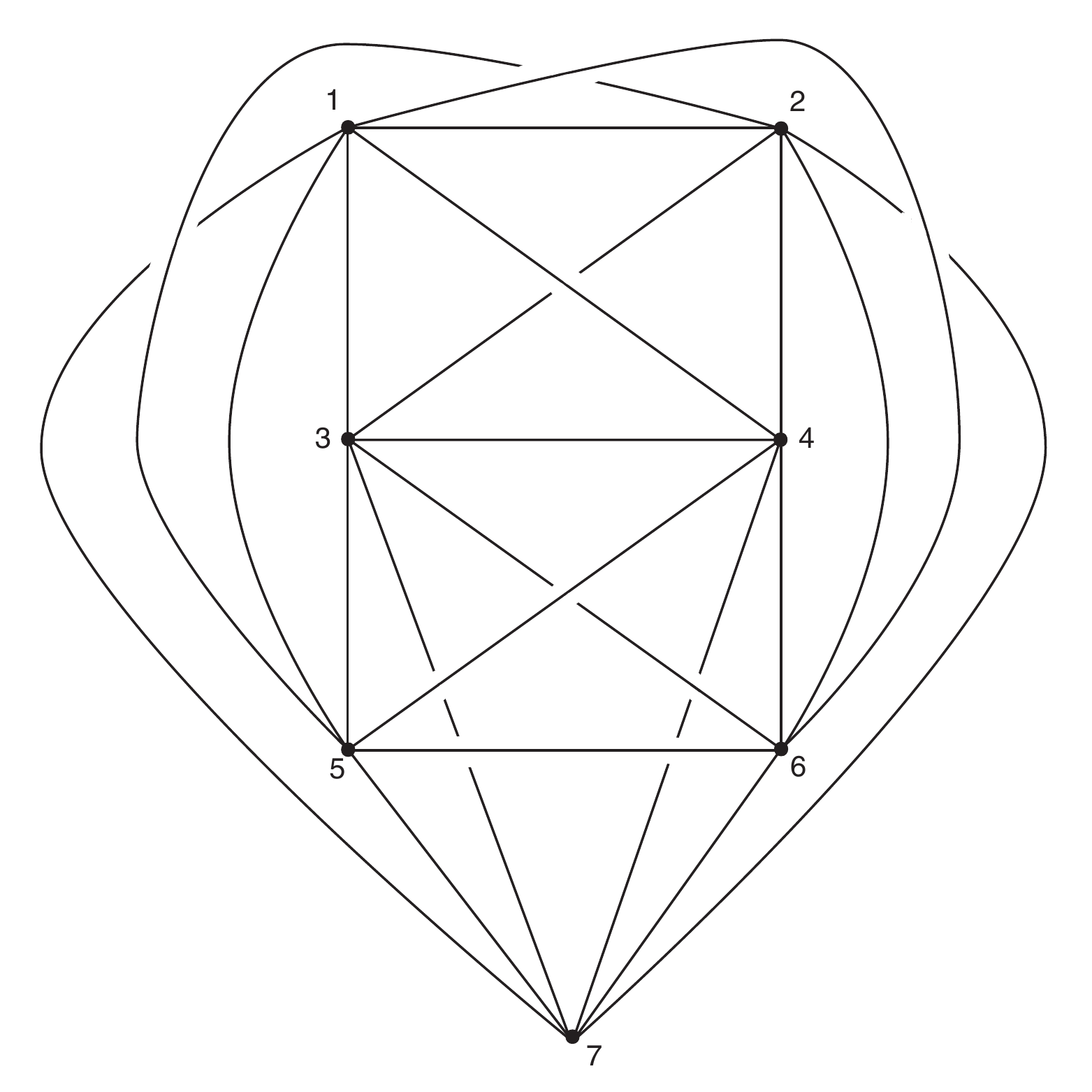}
\caption{The embedding of $K_7$ from \cite{fmellor} with exactly 21 non-split links.} \label{FMellorK7}
\end{figure}

We will demonstrate a linkless embedding in each case, relying primarily upon the embedding from
\cite{fmellor} shown in Figure \ref{FMellorK7} that contains exactly 21 nonsplit links. We will call this the
FMellor embedding.  The links in the FMellor embedding are the following:

457-236 457-136 457-1362 457-1236

147-236 147-235 147-2356 147-2365

167-235 167-245 167-2435 167-2345

136-245 136-2547 136-2457

235-1467 235-1647

245-1376 245-1736

236-1475 236-1547

Suppose $T$ has a vertex $v$ with in degree = 6. This vertex $v$ cannot be contained in a consistently
oriented cycle, so any consistently oriented link must be contained in $T \setminus v$, which is a tournament
on 6 vertices. By Proposition \ref{K_6_nIL_prop}, a tournament on 6 vertices has a linkless embedding.  Thus
$T$ does as well.

Suppose $T$ has a vertex $v$ with in degree = 5. Label this vertex $7$. There is a unique vertex $w$ such
that the edge $7w$ is oriented from $7$ to $w$. Label $w$ as vertex $3$. Then all of the links in the FMellor
embedding have an inconsistently oriented cycle except possibly

136-245

245-1376 245-1736

Label one of the remaining  vertices $2$. Vertex $2$ is adjacent to each of the four remaining unlabeled
vertices, so at least two of these edges must have the same orientation, i.e. either both from $2$ to the
unlabeled vertices, or from the unlabeled vertices to $2$. Label the end point of two edges with matching
orientation $4$ and $5$. Then the cycle $245$ is not consistently oriented, so the FMellor embedding of $T$
contains no consistently oriented nonsplit links.

Suppose $T$ has a vertex $v$ with in degree = 4. Label this vertex $7$. There are four edges oriented from
another vertex to $7$. Label the end point of these edges $1,4,5,6$. Then all of the links in the FMellor
embedding have a cycle with inconsistent orientation except possibly:

136-245 136-2547 136-2457

245-1376 245-1736

If two or more of the edges between $\{1,4,5,6\}$ and $2$ are oriented from $w$ to $2$, then we may rearrange
the labels of $\{1,4,5,6\}$ so that two of these edges are $42$ and $52$.  In this case, $245$ has an
inconsistent orientation, and the cycles $2547$ and $2457$ are inconsistently oriented as well.   Thus, the
FMellor embedding has no consistently oriented non-split links.  As we may switch the labels of vertex $2$
and vertex $3$, at most one edge can be oriented from $\{1,4,5,6\}$ to vertex $3$ as well.

Vertex $2$ has at most one edge oriented from $\{1,4,5,6\}$ to $2$, call it $w2$. Similarly $3$ has at most
one edge oriented from $\{1,4,5,6\}$ to $3$, call it $w'3$.  If no such $w$ or $w'$ exists, or if $w \neq w'$
then we may choose to rearrange the labels $\{1,4,5,6\}$ such that edges $24$ and $25$ are oriented from $2$
to $4$ and from $2$  to $5$, and such that edges $31$ and $36$ are oriented from $3$ to $1$ and from $3$ to
$6$.  Then the cycles 136 and 245 are not consistently oriented and the FMellor embedding has no consistently
oriented non-split links.

So we may assume that $w=w'=4$.   Then cycle 136 is not consistently oriented, and the only possible links in
the FMellor embedding are:

245-1376 245-1736

\begin{figure}
\includegraphics[scale=.6]{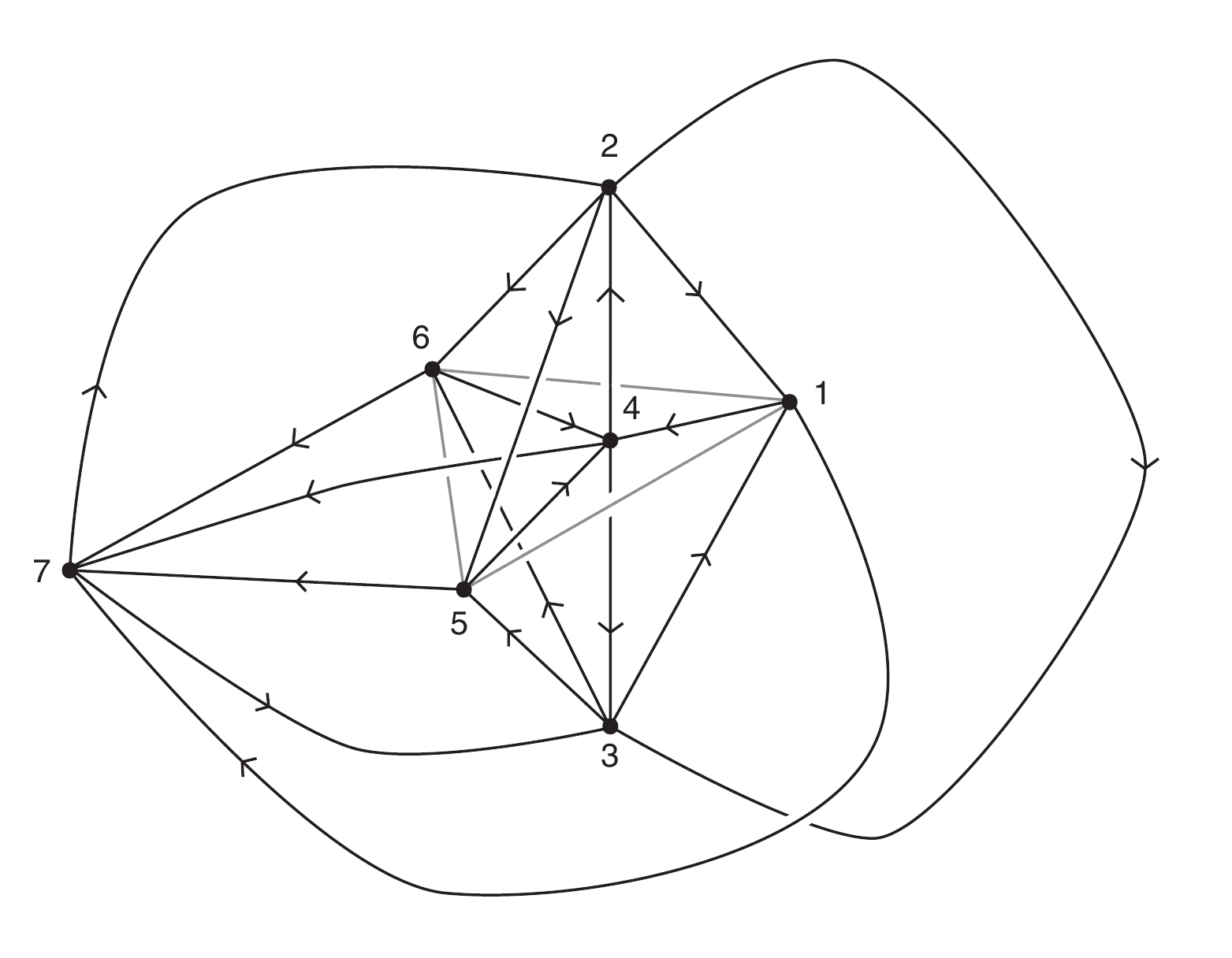}
\caption{An alternate embedding of a directed $K_7$. The gray edges may have any orientation.} \label{AltK7}
\end{figure}

If edge $45$ is oriented from $4$ to $5$, then the cycle 245 is not consistently oriented. So we may assume
that edge $45$ is oriented from $5$ to $4$. As we can exchange labels among $\{1,5,6\}$, we may assume that
edges $41$ and $46$ are oriented to $4$ as well. Thus, we have restricted the orientation of all edges in $T$
except for the triangle $156$, and edge $23$.  Notice that edge $2v$ is oriented from $2$ to $v$ if and only
if $3v$ is oriented from $3$ to $v$ for all $v$, so we may assume edge $23$ is oriented from $2$ to $3$
(exchanging the labels of $2$ and $3$ if necessary).  Consider the embedding $f$ of $T$ shown in Figure
\ref{AltK7}. As any link must be between two 3-cycles or a 3-cycle and a 4-cycle, to show $f(T)$ contains no
consistently oriented non-split links, it suffices to check that each 3-cycle in $f(T)$ is either not
consistently oriented, bounds a disk, or forms a link only with inconsistently oriented cycles.

Any 3-cycle that uses edge $23$ is inconsistently oriented.

Any 3-cycle that contains vertex $7$ is inconsistently oriented or bounds a disk.

Any consistent 3-cycle not using the above bounds a disk, or is $561$.

If $561$ is consistent, all of the cycles with which it forms a non-trivial link are inconsistently
oriented.

Suppose the maximum in degree of a vertex in $T$ is 3.  Note that this implies that all vertices in $T$ have
in degree = out degree = 3.  Choose a vertex, label it $7$. Label the other vertices so that edges $17$, $47$
and $57$ are oriented to $7$.

As each of $6, 2, 3$ have in degree 3, there must be nine edges that terminate on vertices $\{2,3,6\}$. Three
edges from vertex $7$ terminate there, and the three edges that form a $K_3$ on $\{2,3,6\}$ have end points
there as well. Thus, there must be three edges from $\{1,4,5\}$ to $\{2,3,6\}$. Choose one such edge and
rearrange the labels as necessary so that this edge is $16$. Then all of the non-trivial links in the FMellor
embedding have an inconsistently oriented cycle except possibly:

136-245 136-2547 136-2457

235-1467

We break the remainder of the argument into three cases.

Case 1: one or both of $26$ and $36$ are oriented from $2$ to $6$ or $3$ to $6$. As $6$ has in degree 3, and
$76$ and $16$ are oriented to $6$, at most one of $26$ and $36$ are oriented to $6$. Switching labels if
necessary, we may assume edge $36$ is oriented from $3$ to $6$.

Then cycle $136$ is not consistently oriented, so the only remaining potential link is 235-1467. Since $6$
has in degree three, edges $64$, $65$ and $62$ are oriented from $6$ to the other vertex. We have edge $76$
oriented from $7$ to $6$ and edge $64$ oriented from $6$ to $4$.  If $14$ is oriented from $1$ to $4$, then
$1467$ has an inconsistent orientation, so there are no consistently oriented links in the FMellor embedding,
and we are done.  So we may assume $14$ is oriented from $4$ to $1$. As we may switch the labels of vertex
$4$ and vertex $5$, edge $15$ must be oriented from $5$ to $1$ as well.

Similarly, both $234$ and $235$ must be consistently oriented or else we may switch the labels of $4$ and $5$
so that $235$ has an inconsistent orientation, giving a linkless version of the FMellor embedding of $T$. As
$72$ is oriented from $7$ to $2$, and $62$ is oriented from $6$ to $2$, we cannot have both $42$ and $52$
oriented from $4$ and $5$ to $2$.   Thus, we must have $25$ oriented from $2$ to $5$ and $24$ oriented from
$2$ to $4$.  For $234$ and $235$ to be consistently oriented, edge $23$ must be oriented from $3$ to $2$ and
edges $35$ and $34$ must be oriented from $5$ to $3$ and from $4$ to $3$.

We have edges $72$, $62$, and $32$ oriented to vertex $2$.  As vertex $2$ has in degree $3$, this implies
edge $12$ must be oriented from $2$ to $1$.   We have edges $73$, $53$ and $43$ oriented to $3$. As vertex
$3$ has in degree 3, this implies that edge $13$ must be oriented from $3$ to $1$.  This is a contradiction,
as vertex $1$ has in degree 3, but edges $12$, $13$, $14$ and $15$ all terminate at $1$.  Thus, the edge
orientations must be in a configuration that allows a linkless FMellor embedding.

Case 2: We may assume that $26$ is oriented from $6$ to $2$ and $36$ is oriented from $6$ to $3$, as
otherwise we would be in Case 1.  Assume one or both of $12$ and $13$ are oriented from $1$ to $2$ or $1$ to
$3$.  Exchanging labels if necessary, we may assume $13$ is oriented from $1$ to $3$.

Then, the 3-cycle $136$ is not consistently oriented, as edge $36$ is oriented from $6$ to $3$ and edge $13$
is oriented from $1$ to $3$ by assumption. Thus, the only potential link in the FMellor embedding is
235-1467.  As the edge $76$ is oriented from $7$ to $6$, the 4-cycle 1467 is only consistently oriented if
edge $46$ is oriented from $6$ to $4$.  However, as edges $26$ and $36$ are oriented from $6$ to $2$ and from
$6$ to $3$ by assumption, there is exactly one more edge oriented from a vertex $6$ to $v$.  Thus, at least
one of the edges $46$ and $56$ is oriented from $v$ to $6$.  Thus, switching labels of vertices $4$ and $5$
if necessary, we may assume that edge $46$ is oriented from $4$ to $6$, so there are no consistently oriented
links in the FMellor embedding.

Case 3: We may assume that edge $26$ is oriented from $6$ to $2$ and edge $36$ is oriented from $6$ to $3$,
as otherwise we would be in Case 1.   Further, we may assume edge $12$ is oriented from $2$ to $1$ and edge
$13$ is oriented from $3$ to $1$, as otherwise we would be in Case 2.

There are three edges oriented from $\{1,4,5\}$ to $\{6,2,3\}$, one of which is edge $16$. Suppose the other
two are of the form $v6$ and $v'6$, where $v'$ may or may not be equal to $v$.  Then edges $21, 24$ and $25$
are oriented from $2$ to the other vertex, and edges $31, 34$ and $35$ are oriented from $3$ to the other
vertex.  Edge $23$ is oriented from $2$ to $3$ or from $3$ to $2$.  As both vertex $2$ and vertex $3$ have
out degree = 3, either choice of orientation for edge $23$ gives a contradiction.   Thus there exists an edge
$vw$ from $\{1,4,5\}$ to $\{6,2,3\}$ with $w \neq 6$.

As $26$ and $36$ are oriented from $6$ to $2$ and $6$ to $3$ by assumption, edge $w6$ is oriented from $6$ to
$w$.  We may exchange the labels of $w$ and $6$, which gives either $36$ oriented from $3$ to $6$ or $26$
oriented from $2$ to $6$. We may exchange the labels $v$ and $1$ if necessary so that edge $16$ is oriented from $1$ to $6$, and so reduce to Case 1.

\end{proof}

\begin{theorem}
There exists a tournament on 8 vertices that is intrinsically linked as a digraph.
\label{linked_tournament_thm}
\end{theorem}

\begin{proof}
Label the vertices of $K_8$ as $\{a_1, a_2, a_3, b_1, b_2, b_3, x, y\}$.  Consider the subgraph $H$ of $K_8$
isomorphic to $K_{3,3,2}$ formed by choosing the vertex partitions $\{a_1, a_2, a_3\}, \{b_1, b_2, b_3\}$ and
$\{x, y\}$.  Orient the edges of $H$ as follows:
 from $x$ and $y$ to $a_i$, from $b_j$ to $x$ and $y$, and from $a_i$ to $b_j$.

Every embedding of $K_{3,3,2}$ contains a pair of disjoint 3-cycles that have non-zero linking number
\cite{REU}.  As these 3-cycles are disjoint, they must be of the form $xa_ib_j$ and $ya_kb_l$.  By the
construction of $H$, these cycles are consistently oriented.  Choosing an
 arbitrary orientation for all edges of $K_8 \setminus H$ gives a tournament that is intrinsically linked as
 a directed graph.
\end{proof}

\section{Intrinsic Knotting in Tournaments}

\begin{prop} No tournament on 7 vertices is intrinsically knotted as a digraph.
\label{unknotted_7_verts}
\end{prop}

\begin{proof}
By \cite{cg} there is an embedding $f$ of $K_7$ that contains a single knotted cycle $c$, and further, $c$ is
a Hamiltonian cycle. Given an orientation on $K_7$, put it in the Conway-Gordon embedding $f$.

Suppose $c$ is consistently oriented.   Then we may label the vertices of $K_7$ so that $c$ is the cycle
$1234567$.  Due to the symmetry of $K_7$ we may also place $K_7$ into the Conway-Gordon embedding so that $c$
is $1234576$.  In this embedding $c$ has an inconsistent orientation, and hence this embedding contains no
consistently oriented cycle that is a non-trivial knot.  Thus any orientation of $K_7$ yields a tournament
that is not intrinsically knotted as a directed graph.

\end{proof}

\begin{prop}
No tournament on 8 vertices is intrinsically knotted as a digraph. \label{unknotted_prop}
\end{prop}

\begin{proof}

In \cite{mellor2}, Abrams, Mellor and Trott demonstrate an embedding of $K_8$ with exactly 29 knotted cycles,
which is shown in Figure \ref{AMT_K8}. We will refer to this as the AMT embedding.  The knotted cycles in the AMT embedding are:

           1462375
           1462385
           1468237
           1468537
           1472385
           1586237
           1586247
           2468537
           15862347
           15862437
           15836247
           14762385
           14723685
           14672385
           15732648
           14623875
           14623785
           12468537
           14682537
           14685237
           15486237
           12735864
           16427358
           13586247
           13724685
           14568237
           14628357
           14682375
           14723856

\begin{figure}
\includegraphics[scale=.6]{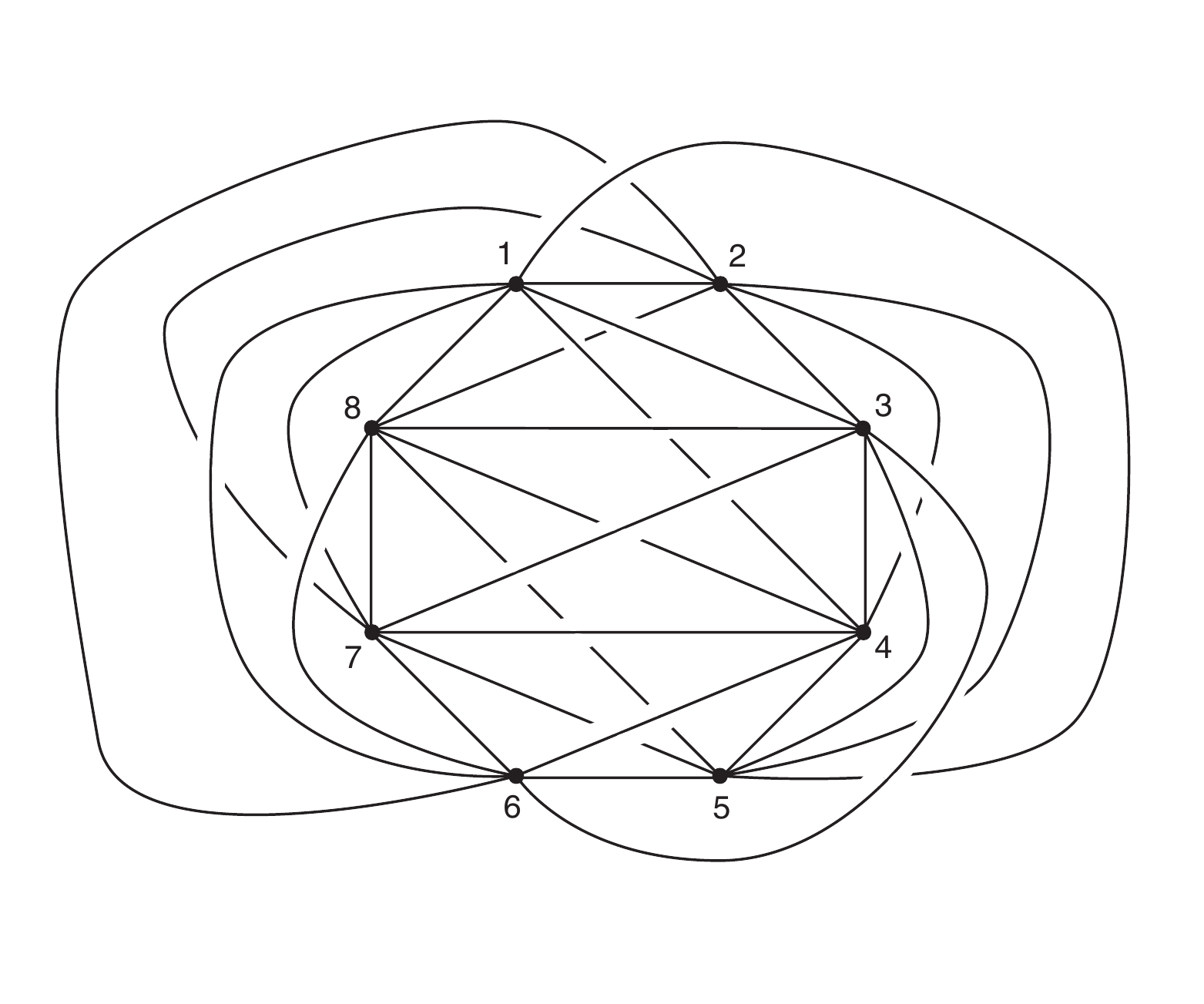}
\caption{The embedding of $K_8$ from \cite{mellor2} that has exactly 29 knotted cycles.} \label{AMT_K8}
\end{figure}

Given a tournament $T$ on 8 vertices, we will show that it can be placed in the AMT embedding such that all
of the knotted cycles are inconsistently oriented. We will break the proof in to cases based on the maximum
in degree of a vertex in the tournament. We need only check the cases of max in degree 7, 6, 5, and 4, as in
the other cases we may repeat the argument using maximum out degree.

Suppose there is a vertex $v$ in $T$ with in degree = 7.   Then no consistently oriented cycle can contain
vertex $v$, and so any consistently oriented cycle must be contained in $T \setminus v$, which is a  tournament on seven vertices.   By Proposition \ref{unknotted_7_verts} a tournament on seven vertices has
an embedding with no consistently oriented cycle that forms a non-trivial knot.  Thus $T$ does as well, and hence $T$ is
not intrinsically knotted as a directed graph.

Suppose that $v$ is a vertex of maximal in degree, with in degree 4, 5, or 6.   Label $v$ as vertex $1$.
There is a vertex $6$ such that the edge $16$ is oriented from $1$ to $6$.  Let $W$ be the set of vertices
where edge $w_i 1$ is oriented from $w_i$ to $1$.  As $1$ has maximal in degree, there exists some $w_i$ such
that the edge $6w_i$ is oriented from $6$ to $w_i$.   Label this $w_i$ as vertex $4$.  There are at least 3
more elements of $W$.    There is an edge between $w_1$ and $w_2$, label them $5$ and $8$ so that the edge
$58$ is oriented from $8$ to $5$. Label one of the remaining elements of $W$ as vertex $7$.

Embed $T$ in the AMT embedding shown in Figure \ref{AMT_K8}. We have edge orientations $41$ $51$ $71$ $81$
$85$ and $64$. With these orientations, all of the knotted cycles in the AMT embedding have inconsistent
orientations, and hence $T$ has an embedding that does not contain a consistently oriented nontrivial knot,
making $T$ not intrinsically knotted as a directed graph.

\end{proof}

One method to prove a graph is intrinsically knotted is to find a $D_4$ graph minor in every embedding that
satisfies certain linking conditions for its cycles \cite{ty} \cite{foisy}.  The authors extended this
technique to directed graphs in \cite{flemfoisy}.  However, in the directed graph case, the $D_4$ must
satisfy certain edge orientation conditions and rather than being a minor, it must be found by deleting edges
and identifying vertices through an operation called \emph{consistent edge contraction}. Vertices $v$ and $w$
may be identified by consistent edge contraction if edge $vw$ is oriented from $v$ to $w$ and either $v$ is a
sink in $G \setminus vw$ or $w$ is a source in $G \setminus vw$.

\begin{prop}
There exists a tournament on 12 vertices that is intrinsically knotted as a digraph. \label{knotted_prop}
\end{prop}

\begin{proof}

We will choose an orientation of the edges of $K_{12}$ so that every embedding of the resulting tournament $T$ allows the construction of the $\overline{D_4}$ graph from Corollary 3.1 of \cite{flemfoisy}, and hence must be intrinsically knotted
as a digraph.

Label the 12 vertices $x_1, x_2, x_3, y_1, y_2,y_3,a_1, a_2,a_3,b_1,b_2,b_3$. Orient the edges from $x_i$ to
$y_j$, from $x_i$ to $x_j$ for $j>i$, and from $y_i$ to $y_j$ for $j>i$. Orient the edges from $y_i$ to
$a_j$, from $x_3$ to $a_i$, from $y_3$ to $b_i$, from $a_i$ to $b_j$, from $b_i$ to $x_j$, and $b_i$ to
$y_1$. Choose arbitrary orientations for the remaining edges.

Fix an embedding of $T$. Consider the subgraph formed by the $x_i$ and $y_j$. The underlying graph is $K_6$,
so it contains a pair of 3-cycles $T_1$, $T_3$ with odd linking number. Up to switching the labels of $T_1$
and $T_3$, there are three cases:

$T_1$ is $x_1x_2x_3$. In this case, $x_1$ is a source in $T_1$, $x_3$ is a sink in $T_1$, $y_1$ is a source
in $T_3$, and $y_3$ is a sink in $T_3$.

$T_1$ is $x_i,y_1,y_2$. In this case, $x_i$ is a source in $T_1$, $y_2$ is a sink in $T_1$, $x_j$ is a source
in $T_3$, and $y_3$ is a sink in $T_3$.

$T_1$ is $x_ix_jy_k$ with $k \neq 3$ and $j>i$.  Then $x_i$ is a source in $T_1$, $y_k$ is a sink in $T_1$,
$x_m$ is a source in $T_3$ and $y_3$ is a sink in $T_3$.

So in all cases, $y_3$ is a sink in $T_3$, the sink of $T_1$ is one of $x_3, y_1, y_2$ and the sources of
$T_1$ and $T_3$ are elements of $\{x_1, x_2,x_3,y_1\}$.

Consider the subgraph formed by the vertices $y_3, a_i, b_j$. This graph contains $K_{3,3,1}$, with $y_3$ as
the preferred vertex, and $a_i$, $b_j$ as the other partitions. Thus the embedding of $T$ contains a 3-cycle
$T_4$ and a 4-cycle $S$ with odd linking number, where the 3-cycle is $y_3a_mb_n$ and the 4-cycle is of the
form $a_ib_ja_kb_l$. In the 3-cycle, $y_3$ is a source and $b_m$ is a sink. In the 4-cycle, the $a$ vertices
have out degree 2, and the $b$ vertices have in degree 2.  There is an edge between $b_j$ and $b_k$, we may
assume it is oriented from $b_j$ to $b_k$. This edge divides $S$ into two 3-cycles, $T_2$ and $T_2'$, at
least one of which has odd linking number with $T_4$. We may assume it is $T_2$. Note that $a_i$ is a source
in $T_2$ and $b_k$ is a sink.

We may now construct the $\overline{D_4}$, from $T_1, T_2, T_3,$ and $T_4$. Note that $y_3$ is the sink of
$T_3$ and the source in $T_4$. The sink of $T_4$ is $b_i$ for some $i$, and the source of $T_1$ is one of
$x_1, x_2, x_3$. By construction, there is an edge oriented from $b_i$ to the source of $T_1$. Similarly, the
sink of $T_1$ is one of $x_3, y_1, y_2$ and the source of $T_2$ is one of the $a_i$, and so by construction
there is an edge from the sink of $T_1$ to the source of $T_2$. The sink of $T_2$ has an edge to any possible
source of $T_3$.

We now delete all edges except those in $T_i$ and the edges we have chosen connecting the $T_i$. Using
consistent edge contraction, we may contract edges to complete the construction of the $\overline{D_4}$. Thus
$T$ contains a consistently oriented nontrivial knot in any embedding, and hence is intrinsically knotted as
a digraph.

\end{proof}

\section{Intrinsic 3-Linking in Tournaments}

\begin{prop}
No tournament on 9 vertices is intrinsically 3-linked as a digraph.
\end{prop}

\begin{proof}
$K_9$ has an embedding with no 3-link \cite{fnp}, so any 9 vertex tournament has an embedding with no
non-split consistently oriented 3 component link.
\end{proof}

\begin{prop} There exists a tournament on 23 vertices that is intrinsically 3-linked as a digraph.
\label{23_3-linked_prop}
\end{prop}

\begin{proof}

Start with $K_{3,3,2}$, with vertex partitions $A=\{a_1,a_2,a_3\}$, $B=\{b_1,b_2,b_3\}$ and $C=\{c_1,c_2\}$,
where edges are directed from $A$ to $B$, from $B$ to $C$, and from $C$ to $A$. Expand vertex
$c_1$ into edge $d_1d_2$, so that vertices in $B$ are connected only to $d_1$, and directed to $d_1$, and
vertices in $A$ are only connected to $d_2$, and such edges are directed from $d_2$ to $A$. Denote this
directed graph $D$. It follows readily from work of Chan et al \cite{REU} that the graph $D$ is intrinsically
linked, and in every spatial embedding, there is a pair of linked cycles $C_1,C_2$ with $C_1$ containing the
edge $d_1d_2$. We similarly expand vertex $c_1$ in a second copy of $K_{3,3,2}$ with the same edge
orientations and denote the resulting graph $D'$.

Glue $D$ and $D'$ along $d_1d_2$ to get $DD$. Now, let $\hat{D}$ be a directed graph that is defined
similarly to $D$, except the edge $e=d_1d_2$ is oriented the other way, from $d_2$ to $d_1$. Glue $\hat{D}$
to $DD$ along $d_1d_2$, so that the glued edge goes from $d_1$ to $d_2$ in the copies of $D$, but from $d_2$
to $d_1$ in $\hat{D}$. Call the resulting graph $DDD$. Embed $DDD$. In each copy of $D$ and $D'$, there is a
pair of consistently oriented links, call them $C_1$ and $C_2$ (and $C_1'$ and $C_2'$ in $D'$), where $C_1$
and $C_1'$ share the edge $d_1d_2$. Similarly, in $\hat{D}$, there is a pair of linked cycles $C_1''$ and
$C_2''$, where $C_2''$ is consistent, $C_1''$ is not, and $C_1''$ shares edge $d_1d_2$. Consider the set of
cycles $S=\{C_1, C_1', C_1''+C_1', C_1+C_1''\}$, then $[C_1]-[C_1']+[C_1'+C_1'']-[C_1+C_1'']=0$ in
$H_1({\mathbb R}^3-C_2, {\mathbb Z})$ and in $H_1({\mathbb R}^3-C_2', {\mathbb Z})$ and in $H_1({\mathbb
R}^3-C''_2, {\mathbb Z})$ (see Figure \ref{sec}). It follows that $C_2, C_2'$ and $C_2''$ each have non-zero
linking number with at least two cycles in $S$. By the pigeonhole principle, there is a non-split 3-link of
consistently oriented cycles.

The digraph $DDD$ has at most one edge between any pair of vertices, so we may add directed edges (of
arbitrary orientation) to form a tournament $T$ on 23 vertices.  As $T$ has $DDD$ as a subgraph, it is
intrinsically 3-linked as a directed graph.

\begin{figure}
\includegraphics[scale=.18]{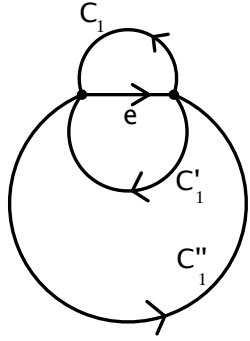}
\caption{The cycle $C_2$ is linked with at least two consistently oriented cycles in this figure.}

\label{sec}
\end{figure}
\end{proof}

\section{Intrinsic 4-Linking in Tournaments}

\begin{prop}
No tournament on 11 vertices is intrinsically 4-linked as a directed graph.
\end{prop}

\begin{proof}
A 4 component link must contain 4 disjoint cycles. Any cycle in a tournament must contain at least 3
vertices. Thus any 4-linked tournament must have 12 or more vertices.
\end{proof}

\begin{prop} There exists a tournament on 66 vertices that is intrinsically 4-linked as a digraph.
\label{4-linked_prop}
\end{prop}

\begin{proof}
Let $H$ be the intrinsically $3$-linked graph constructed in Proposition \ref{23_3-linked_prop}.  In any
embedding of $H$, there is a 3 component link $(L_1, L_2, L_3)$ where $L_1$ is an element of $S=\{C_1, C_1',
C_1''+C_1', C_1+C_1''\}$, and $L_2$ and $L_3$ have non-zero linking number with $L_1$.  Further, $L_1$ must
contain vertices $d_1$ and $d_2$.  As $L_1$ is consistently oriented, we may consider $L_1$ to be composed of
two paths, $P_1$ from $d_1$ to $d_2$, and $P_2$ from $d_2$ to $d_1$.

Let $G$ be a directed graph formed from 3 copies of $H$.  Label them $H_1, H_2,$ and $H_3$.  To form $G$,
identify $d_1$ from $H_1$ to $d_2$ in $H_2$, $d_1$ in $H_2$ to $d_2$ in $H_3$ and $d_1$ in $H_3$ to $d_2$ in
$H_1$.

In any embedding of $G$, we may find a $3$-component link $(L_{i1}, L_{i2}, L_{i3})$ in each $H_i$.  Using
the decomposition of $L_{i1}$ into paths, we may form the consistently oriented cycles $Z = P_{11} \cup
P_{21} \cup P_{31}$ and $W = P_{12} \cup P_{22} \cup P_{32}$.  Let $S' = \{L_{11}, L_{21}, L_{31}, Z,
W\}$ and note that $[L_{11}] + [L_{21}] + [L_{31}] - [Z] - [W] = 0$ in $H_1({\mathbb R}^3-L_{i2}, {\mathbb
Z})$ and in $H_1({\mathbb R}^3-L_{i3}, {\mathbb Z})$.  It follows that $L_{12}, L_{13}, L_{22}, L_{23},
L_{32}$ and $L_{33}$ each have non-zero linking number with at least 2 cycles in $S'$.  By the pigeonhole
principle, there is a non-split 4-link of consistently oriented cycles.

The digraph $G$ has at most one edge between any pair of vertices, so we may add directed edges (of arbitrary
orientation) to form a tournament $T$ on 66 vertices.  As $T$ has $G$ as a subgraph, it is intrinsically
4-linked as a directed graph.

\end{proof}

\section{Intrinsic 5-Linking in Tournaments}

\begin{prop}
No tournament on 14 vertices is intrinsically 5-linked as a directed graph.
\end{prop}

\begin{proof}
A 5 component link must contain 5 disjoint cycles. Any cycle in a tournament must contain at least 3
vertices. Thus any 5-linked tournament must have 15 or more vertices.
\end{proof}

\begin{prop} There exists a tournament on 154 vertices that is intrinsically 5-linked as a digraph.
\label{5-linked_prop}
\end{prop}

\begin{proof}
We proceed as in the 4-linking case.

Let $H$ be the intrinsically $3$-linked graph constructed in Proposition \ref{23_3-linked_prop}.  In any
embedding of $H$, there is a 3 component link $(L_1, L_2, L_3)$ where $L_1$ is an element of $S=\{C_1, C_1',
C_1''+C_1', C_1+C_1''\}$, and $L_2$ and $L_3$ have non-zero linking number with $L_1$.  Further, $L_1$ must
contain vertices $d_1$ and $d_2$.  As $L_1$ is consistently oriented, we may consider $L_1$ to be composed of
two paths, $P_1$ from $d_1$ to $d_2$, and $P_2$ from $d_2$ to $d_1$.

Let $G$ be a directed graph formed from 7 copies of $H$.  Label them $H_1, H_2, \ldots H_7$.  To form $G$,
identify $d_1$ from $H_1$ to $d_2$ in $H_2$, $d_1$ in $H_2$ to $d_2$ in $H_3$ and so on to $d_1$ in $H_7$ to
$d_2$ in $H_1$.

In any embedding of $G$, we may find a $3$-component link $(L_{i1}, L_{i2}, L_{i3})$ in each $H_i$.  Using
the decomposition of $L_{i1}$ into paths, we may form the consistently oriented cycles $Z = P_{11} \cup
P_{21} \cup \ldots \cup P_{71}$ and $W = P_{12} \cup P_{22} \cup \ldots \cup P_{72}$.  Let $S' = \{L_{11},
L_{21}, \ldots L_{71}, Z, W\}$ and note that $[L_{11}] + [L_{21}] + \ldots + [L_{71}] - [Z] - [W] =
0$ in $H_1({\mathbb R}^3-L_{i2}, {\mathbb Z})$ and in $H_1({\mathbb R}^3-L_{i3}, {\mathbb Z})$.  It follows
that $L_{12}, L_{13}, L_{22}, L_{23}, \ldots L_{72}$ and $L_{73}$ each have non-zero linking number with at
least 2 cycles in $S'$.  By the pigeonhole principle, there is a non-split 5-link of consistently oriented
cycles.

The digraph $G$ has at most one edge between any pair of vertices, so we may add directed edges (of arbitrary orientation) to form a tournament $T$ on 154 vertices.  As $T$ has $G$ as a subgraph, it is intrinsically
5-linked as a directed graph.

\end{proof}

\section{Intrinsic N-linking in Tournaments}
\label{n_link_section}

For the general case, we will rely on the construction from Lemma 1 of Flapan, Mellor and Naimi
\cite{flapan}.  We first note the obvious lower bound.

\begin{prop}
No tournament on $3n-1$ vertices is intrinsically $n$-linked as a directed graph.
\end{prop}

\begin{proof}
An $n$-component link must contain $n$ disjoint cycles. Any cycle in a tournament must contain at least 3
vertices. Thus any $n$-linked tournament must have $3n$ or more vertices.
\end{proof}

The following lemma establishes the building block we need for the full construction.

\begin{lemma}
There exists a tournament $T'$ on 8 vertices such that every embedding of $T'$ contains a link $(L_1, L_2)$
such that the linking number $lk(L_1,L_2)$ is odd, $L_2$ is consistently oriented, and $L_1$ contains two
vertices $a$ and $b$ such that it may be decomposed into two paths $P_1$ and $P_2$ such that each $P_i$ is
consistently oriented from vertex $a$ to vertex $b$. \label{t_prime_lemma}
\end{lemma}

\begin{proof}

Start with $K_{3,3,2}$ with vertex partitions $A=\{a_1,a_2,a_3\}$, $B=\{b_1,b_2,b_3\}$ and $C=\{c_1,c_2\}$,
where all edges are directed from $A$ to $B$, from $B$ to $c_2$, from $c_2$ to $A$, from $A$ to $c_1$ and
$c_1$ to $B$. Give all remaining edges an arbitrary orientation, and label this tournament $T'$.

As every embedding of $K_{3,3,2}$ contains a pair of 3-cycles with odd linking number, so does every
embedding of $T'$.  Each 3-cycle must contain exactly one of the $c_i$.  Label the 3-cycle that contains
$c_i$ as $L_i$.   Then $lk(L_1,L_2)$ is odd, and $L_2$ is consistently oriented.  The cycle $L_1$ contains
the vertices $c_1, a_j, b_k$, so it may be decomposed into the path $P_1 = a_j b_k$  and the path $P_2 = a_j
c_1 b_k$, each of which are consistently oriented from $a_j$ to $b_k$.

\end{proof}

We now construct the $n$-linked tournament.

\begin{prop}
There exists a tournament $T$ on $8(2n-3)^2$ vertices such that $T$ is intrinsically $n$-linked as a
directed graph. \label{nlinking_prop}
\end{prop}

\begin{proof}

We begin with $(2n-3)^2$ copies of the tournament $T'$ from Lemma \ref{t_prime_lemma} and label them $T'_i$.
We add edges oriented from $b_{ik}$ to $a_{(i+1)j}$ and edges oriented from $b_{(2n-3)^2k}$ to $a_{1j}$.  We
then add edges of arbitrary orientation to complete the construction of the tournament $T$.

Embed $T$. We may now follow the proof of Lemma 1 from \cite{flapan}.   In each copy of $T'_i$, we may find a
2-link $L_{i1}, L_{i2}$ and edges $e_i$ from $L_{i1}$ to $L_{(i+1)1}$.  Let $C = \bigcup P_{i1} \cup e_i$.
Notice that $C$ is consistently oriented. If $lk( C, L_{i2}) \neq 0$ for $n-1$ cycles $L_{i2}$, we are done.
If not, we will construct an index set $I$ and form a new cycle $Z = \bigcup P_{i\epsilon} \cup e_i$ where
$\epsilon = 2$ if $i \in I$ and $\epsilon = 1$ otherwise.  Note that for any choice of $I$, the cycle $Z$ is
consistently oriented.

Let $M$ be a $(2n-3)^2$ by $(2n-3)^2$ matrix where the entry $m_{ij} = lk(L_{i1}, L_{j2})$ modulo 2.  By the
construction of $T'$, $m_{ii} = 1$.  Let $M'$ be the reduced row echelon form of $M$ modulo 2.  Let $r$ be
the rank of $M$.   If $r \geq (2n-3)$ then let $V$ be the modulo 2 sum of the rows of $M'$.   If $r <
(2n-3)$, then as each column of $M$ contains a 1, some row of $M'$ contains at least $2n-3$ non-zero entries.
Let $V$ be this row.

In either case, $V$ can be written as the modulo 2 sum of rows of $M$, so $V = \sum_{i \in I} r_i$, where
$r_i$ are the rows of $M$.  Let $Z = \bigcup P_{i\epsilon} \cup e_i$ where $\epsilon = 2$ if $i \in I$ and
$\epsilon = 1$ otherwise.  Let $V_j$ be the $j$th entry of $V$. Notice that that $V_j \equiv \sum_{i \in I}
lk(L_{i1}, L_{j2})$ modulo 2 and that $lk(Z, L_{j2}) \equiv lk(C, L_{j2}) + \sum_{i \in I} lk(L_{i1},
L_{j2})$ modulo 2.

Thus, $lk(Z, L_{j2}) \equiv lk(C, L_{j2}) + V_j$ modulo 2.  At least $2n-3$ of the $V_j$ are odd, and there exist
at most $n-2$ components $L_{j2}$ that have non-zero linking number with $C$.  Thus, $lk(Z, L_{j2}) \equiv 1$
modulo 2 for at least $n-1$ components $L_{j2}$, giving a non-split link with $n$ consistently oriented
components.

\end{proof}

Flapan, Mellor and Niami introduced the idea of \emph{linking patterns} in the study of intrinsically linked
graphs in \cite{flapan}.  The linking pattern of a link $L_1, \ldots L_n$ is the graph $\Gamma$ with vertices
$v_1 \ldots v_n$ and an edge between $v_i$ and $v_j$ if $lk(L_i,L_j) \neq 0$.  They then show that for any
linking pattern $\Gamma$ there exists a graph $G$ such that every embedding of $G$ contains a link whose
linking pattern contains $\Gamma$.  The first step to obtain this general result is to show that for any $n$, there
exists a graph $G$ such that every embedding of $G$ contains a link whose linking pattern contains $K_{n,n}$.
This result requires only the iterative application of Lemma 1 of \cite{flapan}.  As Proposition
\ref{nlinking_prop} is the direct analogue of Lemma 1 of \cite{flapan} for tournaments, we have the analogous
result as a corollary.

\begin{cor}
For all $n$, there exists a tournament $T$ such that every embedding of $T$ contains a consistently oriented
link whose mod 2 linking pattern contains $K_{n,n}$.
\end{cor}

It is likely that the techniques of Mattman, Naimi and Pagano \cite{mnp} can be extended to find examples of
tournaments with arbitrary linking patterns.

\section{Intrinsic Linking with Knotted Components}

If every embedding of a graph $G$ contains a non-split $n$-component link where at least $m$ components of
the link are non-trivial knots, we will say that $G$ is \emph{intrinsically n-linked with m-knotted
components.}

If every embedding of a directed graph contains a consistently oriented non-split $n$-component link where at
least $m$ components of the link are non-trivial knots, we will say that $G$ is \emph{intrinsically n-linked
with m-knotted components as a directed graph.}

In \cite{flem}, the first author demonstrated graphs that are intrinsically $n$-linked with $m$-knotted
components for $m<n$. Flapan, Naimi and Mellor used different techniques to construct examples for all $m$,
including $m=n$ in \cite{flapan}.

We will use the techniques of \cite{flem} to construct a tournament that is intrinsically 2-linked with
1-knotted component as a directed graph.

\begin{prop}
No tournament on 8 vertices is intrinsically 2-linked with 1-knotted component as a directed graph.
\end{prop}

\begin{proof}
If every embedding of a tournament contains a consistently oriented non-split link where at least one
component of the link is a non-trivial knot, then that tournament is intrinsically knotted as a digraph.

By Proposition \ref{unknotted_prop}, no tournament on 8 vertices is intrinsically knotted as a digraph.
\end{proof}

\begin{lemma}
There exists a tournament $T'$ on 14 vertices that is intrinsically knotted as a digraph, and further the
pair of adjacent edges $y_3 \alpha$, $\alpha y'_3$ are contained in a consistently oriented non-trivial knot
in every embedding of $T'$. \label{link_knot_t_prime}
\end{lemma}

\begin{proof}

We construct $T'$ based on the knotted tournament from Proposition \ref{knotted_prop}.

The tournament $T'$ contains 14 vertices.  Label them $x_1, x_2, x_3, y_1, y_2, y_3, \alpha, y'_3, a_1,$ $
a_2, a_3, b_1, b_2, b_3$. Orient the edge from $y_3$ to $\alpha$ and the edge from $\alpha$ to $y'_3$.
Orient the edges from $y'_3$ to $a_j$, and from $y'_3$ to $b_i$.

Continue to orient the edges as in Proposition \ref{knotted_prop}, that is from $x_i$ to $y_j$, from $x_i$ to
$x_j$ for $j>i$, and from $y_i$ to $y_j$ for $j>i$. Orient the edges from $y_i$ to $a_j$, from $x_3$ to
$a_i$, from $a_i$ to $b_j$, from $b_i$ to $x_j$, and $b_i$ to $y_1$. Give the remaining edges an arbitrary orientation.

Note that $x_1, x_2, x_3, y_1, y_2, y_3$ form a $K_6$ subgraph, and hence must contain a pair of 3-cycles
with non-zero linking number in every embedding.   Label the $3$-cycle containing $y_3$ as $T_3$.  The vertices
$y'_3, a_1, a_2,a_3,b_1,b_2,b_3$ form a subgraph isomorphic to $K_{3,3,1}$, and hence contain a 3-cycle and
4-cycle with non-zero linking number in every embedding, with $y'_3$ in the 3-cycle.   Label this 3-cycle
$T_4$. Note that $y_3$ is a sink in $T_3$ and $y'_3$ is a source in $T_4$.

We may now continue as in Proposition \ref{knotted_prop}, using consistent edge contraction to construct a
$\overline{D_4}$ graph.  If we choose to contract all edges except $y_3 \alpha$ and $\alpha y'_3$, we obtain
a $\overline{D_4}$ graph, with one vertex expanded into a consistently oriented 2-path.   Thus, the edges
$y_3 \alpha$ and $\alpha y'_3$ are contained in a consistently oriented cycle that is a non-trivial knot in
every embedding of $T'$.

\end{proof}

\begin{prop}
There exists a tournament $T$ on 107 vertices that is intrinsically 2-linked with 1-knotted component as a
directed graph.
\end{prop}

\begin{proof}
We will construct $T$ from 9 copies of the tournament $T'$ from Lemma \ref{link_knot_t_prime} and an
additional vertex $\beta$, in a manner similar to the construction of Proposition 2.2 of \cite{flem}.

Label the 9 copies of $T'$ as $T'_i$.  Identify all $\alpha_i$ to form a single vertex $\alpha$.  We will
identify edges $y_{3i} \alpha$ and $\alpha y'_{3i}$ as follows.  Identify $y_{31} \alpha, y_{32} \alpha,
y_{33} \alpha$ and label the new vertex $v_1$. Identify $y_{34} \alpha, y_{35} \alpha, y_{36} \alpha$ and
label the new vertex $v_2$. Identify $y_{37} \alpha, y_{38} \alpha, y_{39} \alpha$ and label the new vertex
$v_3$.

Identify $\alpha y'_{31}, \alpha y'_{34}, \alpha y'_{37}$ and label the new vertex $w_1$. Identify $\alpha
y'_{32}, \alpha y'_{35}, \alpha y'_{38}$ and label the new vertex $w_2$. Identify $\alpha y'_{33}, \alpha
y'_{36}, \alpha y'_{39}$ and label the new vertex $w_3$.

Orient the edges from $\beta$ to the $w_i$ and from $v_i$ to $\beta$.  Add edges of arbitrary orientation to complete the construction of $T$.

Embed $T$.  In each copy of $T'_i$ there exists a consistently oriented knotted cycle $c_i$  that contains
edges $y_{3i} \alpha$ and $\alpha y'_{3i}$.  The cycle $c_i$ is composed of $y_{3i} \alpha, \alpha y'_{3i}$
and a path $P_i$ from $y'_{3i} $ to $y_{3i}$.

Consider the vertices $\alpha, \beta, v_i, w_i$ together with the edges $\alpha w_i$, $v_i \alpha$, $\beta
w_i$, $v_i \beta$ and the paths $P_i$.  This graph is isomorphic to $K_{3,3,2}$, and thus must contain a pair
3-cycles with non-zero linking number \cite{REU}.  Further, the 3-cycles must be of the form $\alpha w_i v_j$
and $\beta w_k v_l$.  These cycles are consistently oriented, and by construction $\alpha w_i v_j$ is a
non-trivial knot.

\end{proof}

\section{The Disjoint Linking Property in Tournaments}

Chan et al (\cite{REU}) demonstrated graphs that have a disjoint pair of links in every spatial embedding
(have the disjoint linking property) but do not contain disjoint copies of intrinsically linked graphs. We
modify their construction here for tournaments.

\begin{prop}
No tournament on 11 vertices has the disjoint linking property.
\end{prop}

\begin{proof}
A disjoint pair of links must have at least 4 disjoint cycles. Any cycle in a tournament must contain at
least 3 vertices. Thus any tournament with the disjoint linking property must have 12 or more vertices.
\end{proof}

\begin{prop}
There exists a tournament on 14 vertices that has the disjoint linking property, but does not contain two
disjoint tournaments that are intrinsically linked.
\end{prop}
\begin{proof}

Consider $K_{5,5,4}$ with vertex set $A=\{a_1,..., a_5\}$, $B=\{b_1,...,b_5\}$ and $C=\{c_1,c_2, c_3, c_4\}$.
Form a directed graph by orienting edges from $C$ to $A$, from $A$ to $B$, and from $B$ to $C$. Call the
resulting graph $G$. Embed $G$. Let $D$ be the subgraph of $G$ formed by $\{a_1, a_2, a_3, b_1, b_2, b_3,
c_1, c_2\}$.  Note that $D$ is an orientation of $K_{3,3,2}$, and hence contains a pair of 3-cycles with
non-zero linking number. Further, these 3-cycles are consistently oriented. Since the graph $D$ is a subgraph
of $G$, $G$ has a pair of consistently-oriented linked 3-cycles. Removing the linked 3-cycles and all edges
incident to their vertices results in a embedded copy of $D$, which will again have a disjoint pair of linked
3-cycles. Thus $G$ has the disjoint linking property.

As $G$ has at most one edge between any pair of vertices, we may add edges of arbitrary orientation to $G$ to
from a tournament $T$ on 14 vertices.   As $G$ is a subgraph of $T$, $T$ has the disjoint linking property.
Since no tournament on 7 (or fewer) vertices is intrinsically linked, $T$ cannot contain two disjoint copies
of intrinsically linked tournaments.

\end{proof}

\section{The Consistency Gap}

Given a tournament, we may ignore the edge orientations and consider it as an undirected complete graph.  The
complete graph $K_6$ is intrinsically linked \cite{sachs} \cite{cg}, but by Theorem
\ref{linked_tournament_thm}, the smallest tournament that is intrinsically linked as a digraph has 8
vertices.  This implies that every embedding of a tournament on 6 or 7 vertices contains one or more
non-split links, but in some embeddings one or more the components of each non-split link has inconsistent
edge orientations.

Thus the requirement that the components of the non-split link have a consistent orientation for an
intrinsically linked digraph is restrictive and appears to require a larger and more complex graph to
satisfy.  We propose to measure how restrictive this condition is using the \emph{consistency gap}.

We define the \emph{consistency gap}, denoted $cg(n)$, as $m' - m$ where $K_m$ is the smallest intrinsically
$n$-linked complete graph, and $m'$ is the number of vertices in the smallest tournament that is
intrinsically $n$-linked as a directed graph.   Note that $cg(n) \geq 0$ for all $n$.  By the preceding
discussion, we have the following corollary.

\begin{cor}
The consistency gap at 2, $cg(2) = 2$.
\end{cor}

By work of Flapan, Naimi, and Pommersheim, the smallest intrinsically 3-linked complete graph is $K_{10}$
\cite{fnp}.  Combining that with Proposition \ref{23_3-linked_prop}, we have the following.

\begin{cor}
The consistency gap at 3, $cg(3) \leq 13$.
\end{cor}

An $n$-linked complete graph must contain at least $3n$ vertices.   Combining this with Proposition
\ref{4-linked_prop} and Proposition \ref{5-linked_prop} we have the following.

\begin{cor}
The consistency gap at 4, $cg(4) \leq 54$, and the consistency gap at 5, $cg(5) \leq 139$.
\end{cor}

Finally, the results of Section \ref{n_link_section} give a bound for the general case.

\begin{cor}
For $n > 5$, the consistency gap at $n$, $cg(n) \leq 8(2n-3)^2 - 3n$.
\end{cor}

We conjecture the consistency gap to be non-decreasing in $n$.  Does it increase without bound?

\bigskip

\textsc{666 5th Avenue, 8th Floor,  New York, NY 10103}

\medskip

\textsc{Department of Mathematics, SUNY Potsdam,  Potsdam, NY 13676}

\end{document}